\documentclass[12pt,a4paper]{amsart}
\calclayout
\numberwithin{equation}{section}
\allowdisplaybreaks
\theoremstyle{plain}

\newtheorem{theorem}{Theorem}[section]
\newtheorem{lemma}[theorem]{Lemma}

\newtheorem{corollary}[theorem]{Corollary}
\newtheorem{remark}{Remark}

\numberwithin{equation}{section}

\usepackage{enumerate}

\usepackage{amssymb}
\usepackage{mathtools}
\mathtoolsset{showonlyrefs}
\usepackage{mathrsfs}
\usepackage{comment}
\usepackage{hyperref}
\usepackage{xcolor}

\def\P{\mathbb{P} }
\def\R{\mathbb{R} }
\def\d{\mathrm{d}}
\def\D{\mathcal{D} }
\def\E{\mathbb{E} }
\def\e{\mathcal{E}}
\def\F{\mathcal{F}}
\def\Z{\widetilde{Z}}
\def\N{\widetilde{N}}
\def\V{\widetilde{V}}
\def\1{\mathbf{1}}
\def\d{\mathrm{d}}

\begin{document}
	\title[Branching Brownian motion with absorption]{The extremal process of branching Brownian motion with absorption}
	\thanks{The research of this project is supported by the National Key R\&D Program of China (No. 2020YFA0712900).}
	\author[F. Yang and Y. Zhu]{Fan Yang and Yaping Zhu}
	\address{Fan Yang\\ School of Mathematical Sciences \\ Beijing Normal University\\ Beijing 100875\\ P. R. China}
	\email{fan-yang@bnu.edu.cn}
	\thanks{The research of F. Yang is supported by China Postdoctoral Science Foundation (No. 2023TQ0033)}
	\address{Yaping Zhu\\ School of Mathematical Sciences \\ Peking University\\ Beijing 100871\\ P. R. China}
	\email{zhuyp@pku.edu.cn}
	\begin{abstract} 
		In this paper, we study branching Brownian motion with absorption, in which particles undergo Brownian motions with drift and are killed upon reaching the origin.
		We prove that the extremal process of this branching Brownian motion with absorption converges to a random shifted decorated Poisson point process. Furthermore, we show that the law of the right-most particle converges to the law of a random shifted Gumbel random variable.
	\end{abstract}
	
	\maketitle
	\section{Introduction}\label{Sec1}
	\subsection{Background}
	A classical branching Brownian motion (BBM) in $\R$ can be constructed as follows. Initially there is a single particle at the origin of the real line and this particle moves as a 1-dimensional standard Brownian motion $B = \{B(t), t\geq 0 \}$. 
	After an independent exponential time with parameter $\beta$, the initial particle dies and produces $L$ offspring. $L$ is a positive integer-valued random variable with $m: = \E L\in (1,\infty)$.
	Starting from their positions of creation, each of these particles evolves independently and according to the same law as their parent. We denote by $N_t$ the collection of particles alive at time $t$. For any $u\in N_t$ and $s\le t$, let $X_u(s)$ be the position at time $s$ of particle $u$ or its ancestor alive at that time. The maximum of the BBM at time $t$ is defined as $M_t:=\max\{X_u(t):u\in N_t\}$. 
	
	McKean \cite{McKean75} established the connection between BBM and the Fisher-Kolmogorov-Petrovskii-Piskounov (F-KPP)
	reaction-diffusion equation
	\begin{equation}\label{KPPeq1}
		\frac{\partial u}{\partial t} = \frac{1}{2} \frac{\partial^2 u}{\partial x^2} + \beta ({f}({u})-{u}),
	\end{equation}
	where $f(s)=\E(s^{L})$ and $u:\R_+ \times \R \rightarrow [0,1]$.
	More precisely, it is shown in \cite{McKean75} that, for any $[0, 1]$-valued function $g$ on $\R$, $u(t,x) = \E \left[\prod_{v\in N_t} g(x+X_v(t)) \right]$ is a solution of \eqref{KPPeq1} with initial condition $u(0,x)=g(x)$.
	The F-KPP equation has been studied intensively by both analytic techniques (see, for example, Kolmogorov et al. \cite{Kolmogorov37} and Fisher \cite{Fisher37}) and probabilistic methods
	(see, for instance, McKean \cite{McKean75}, Bramson \cite{Bramson78,Bramson83}, Harris \cite{Harris99} and Kyprianou \cite{Kyprianou04}). 
	
	Define
	\begin{equation}\label{def_lambda*}
		\lambda^* := \sqrt{2\beta(m-1)}
	\end{equation} 
	and
	\begin{equation}
		m_t:=\lambda^*t-\frac{3}{2\lambda^*}\log t.
	\end{equation}
    In the case of $m=2$ and $\beta=1$,
	Bramson \cite{Bramson78} established that
	\begin{equation}
			\lim_{t\to\infty}\P(M_t\le m_t+z)=\lim_{t\to \infty} u(t,m_t+z)=w(z),\quad z\in \R,
	\end{equation}
	where $w$ solves the ordinary differential equation
	\begin{equation*}
		\frac{1}{2}w''+\sqrt{2}w'+f(w)-w=0.
	\end{equation*}
	Define 
	\begin{equation}\label{def_derivative}
		Z_t = \sum_{u\in N_t} (\lambda^*t - X_u(t)) e^{\lambda^*(X_u(t)-\lambda^*t) },
	\end{equation}
	then $Z_t$ is known as the derivative martingale of the BBM, see Kyprianou \cite{Kyprianou04}.
	Lalley and Sellke \cite{Lalley87} provided the following representation of $w$ for dyadic BBM
	\begin{equation}\label{travelling}
		w(z):=\E\left[e^{-C_* e^{-\lambda^* z}Z_{\infty}}\right],
	\end{equation}
	where $C_*$ is a positive constant and $Z_{\infty}:=\lim_{t\to \infty}Z_t$ $\P$-almost surely.
	The behavior of the particles at the tip of BBMs was investigated by A\"{i}d\'{e}kon, Berestycki, Brunet and Shi \cite{ABBS13} as well as Arguin, Bovier and Kistler \cite{ABK13}. They considered the extremal process of BBM, which is defined by
	\begin{equation}
		\sum_{u\in N_t} \delta_{X_u(t)-m_t},
	\end{equation}
	and showed that it converges in law to a random shifted decorated Poisson point process (DPPP). 
	A DPPP $\e$ is determined by two components: an intensity measure $\mu$ which is a (random) measure on $\R$, and a decoration process. Conditioned on $\mu$, let $\sum_{i} \delta_{p_i}$ be a Poisson point process with intensity $\mu$, and let $\{\sum_j \delta_{d^i_j}\}$ be a family of independent point processes with law $\D$. Then $\e = \sum_{i,j} \delta_{p_i+d^i_j}$ is a DPPP with intensity $\mu$ and decoration $\D$, denoted by DPPP$(\mu, \D)$. A\"{i}d\'{e}kon et al. \cite{ABBS13} and Arguin et al. \cite{ABK13} obtained that 
	\begin{equation}
		\lim_{t\rightarrow\infty} \sum_{u\in N_t} \delta_{X_u(t)-m_t} = \mathrm{DPPP}\left(\lambda^*C_*Z_{\infty}e^{-\lambda^*x}\d x, \D^{\lambda^*}\right) \mbox{ in law },
	\end{equation}
	where $C_*$ is the positive constant given by \eqref{travelling} and 
	\begin{equation}\label{def_D}
		\D^{\lambda^*}(\cdot) := \lim_{t\rightarrow\infty} \P\left( \sum_{u\in N_t} \delta_{X_u(t)-M_t} \in \cdot \, \Big| M_t \geq \lambda^*t \right).
	\end{equation}
	In the recent paper \cite{Kim23}, Kim, Lubetzky and Zeitouni
	studied the maximum of BBM in $\R^d$. For irreducible multitype branching Brownian motion, we refer the readers to Hou et al. \cite{Hou23} and for reducible multitype branching Brownian motion, we refer the readers to Belloum et al. \cite{Belloum21} and Ma et al. \cite{Ma23}.

	In this paper, we consider the extremal process of BBM with absorption.	We will focus on a branching Brownian motion with 	drift $-\rho$, in which particles are absorbed at the origin. 
	The process can be defined as follows. Starting with a single particle at $x>0$, this particle moves according to a 1-dimensional Brownian motion with drift $-\rho$ $(\rho \in \R)$ until an independent exponentially distributed time with rate $\beta$. 
	When the initial particle dies, it produces a random number $L\ge 1$ particles at the place of its death. These offspring particles evolve independently from their birth place, according the same law as their parent. 
	Assume that $L$ has distribution $\{p_k,k\geq 1\}$ with $\E L = m\in(1,\infty)$ and $\E L^2 < \infty$.
	We add an absorbing barrier at the origin, i.e. particles hitting the barrier are instantly killed without producing offspring. The set $\widetilde{N}_t^{-\rho}$ denotes the particles of the BBM with absorption alive at time $t$. 
	For any $u\in \widetilde{N}_t^{-\rho}$ and $s\leq t$, we still use $X_u(s)$ to denote the position of the particle $u\in \widetilde{N}_t^{-\rho}$ or its ancestor at time $s$. Define $$\widetilde{Y}_t^{-\rho} := \sum_{u\in\widetilde{N}_t^{-\rho}} \delta_{X_u(t)},$$ which is a point process describing the number and positions of individuals alive at time $t$. 
	The extinction time of the BBM with absorption is defined as
	$$\zeta^{-\rho}:=\inf\{t>0: \widetilde{N}_t^{-\rho} = \emptyset \}.$$ 
	Let $\widetilde{M}_t^{-\rho}:=\max\{X_u(t):u\in \widetilde{N}_t^{-\rho}\}$ be the right-most position of the particles in $\widetilde{N}_t^{-\rho}$.
	The law of the BBM with absorption starting from single particle at $x$ is denoted by $\mathbf{P}_x$ and its expectation is denoted by $\mathbf{E}_x$.

	The asymptotic behavior of branching Brownian motion (BBM) with absorption has been studied extensively in the literature.
	Kesten \cite{Kesten78} proved that the process dies out almost surely when $\rho \ge \lambda^*$ while there is a positive probability of survival when $\rho < \lambda^*$. Therefore, $\rho = \lambda^* $ is the critical drift separating the supercritical case $\rho < \lambda^*$ and the subcritical $\rho > \lambda^*$. 
	In the critical case, Kesten \cite{Kesten78} obtained upper and lower bounds on the survival probability, which were improved by Berestycki et al. \cite{Berestycki14}. Maillard and Schweinsberg \cite{Maillard22} have further improved these results.
	Harris et al. \cite{Harris06} studied properties of the right-most particle
	and provided a probabilistic proof of the classical result on the one-sided F-KPP traveling wave solution of speed $-\rho$ in the supercritical case. They proved that
	\begin{equation}\label{growth rate}
		\lim_{t\to \infty} \frac{\widetilde{M}_t^{-\rho}}{t}=\lambda^*-\rho\quad on \ \{\zeta^{-\rho}=\infty\}, \ \mathbf{P}_x\mbox{-a.s.}
	\end{equation}
	and $g(x):=\mathbf{P}_x(\zeta^{-\rho}<\infty )$ is the unique solution to 
	\begin{equation}
		\left \{
		\begin{aligned}
			& \frac{1}{2}g''-\rho g'+ f(g) -g = 0,\quad x>0, \\
			& g(0+)=1,\quad g(\infty)=0,\\
		\end{aligned}\right.
	\end{equation}
	where $f(s) = \E(s^L)$.
	In the subcritical case, the large time asymptotic behavior for the survival probability was given by Harris and Harris \cite{Harris07}. For the BBM with absorption in the near-critical case, Berestycki et al. \cite{Berestycki11} and Liu \cite{Liu21} are good references.
	
	In this paper, we study the extremal process of one-dimensional branching Brownian motions with absorption and prove that the limit of this point process converges to a random shifted decorated Poisson point process.

	\subsection{Main results}
	Define 
	\begin{equation}
		\Z_t^{-\rho} := \sum_{u\in \widetilde{N}_t^{-\rho}} ((\lambda^*-\rho) t-X_u(t)) e^{\lambda^*\left(X_u(t)-(\lambda^*-\rho)t\right)},
	\end{equation}
	and
	\begin{equation}
		m_t^{-\rho}:= (\lambda^*-\rho)t - \frac{3}{2\lambda^*}\log t.
	\end{equation}
	Our first main result, which is Theorem \ref{thrm1} below, deals with the convergence of the process $\Z_t^{-\rho}$ as $t\to \infty$.
	
	\begin{theorem}\label{thrm1}
		For any $x>0$ and $\rho < \lambda^*$, the limit $\Z_{\infty}^{-\rho} := \lim_{t\rightarrow\infty} \Z_t^{-\rho} $ exists $\mathbf{P}_x$-almost surely. Moreover, the events $\{\Z_{\infty}^{-\rho}>0 \}$ and $\{\zeta^{-\rho} = \infty \}$ agree up to a $\mathbf{P}_x$-null set.
	\end{theorem}
	
	Define the extremal process of the BBM with absorption by
	\begin{equation}\label{def_e_t^rho}
		\e_t^{-\rho} := \sum_{u\in \widetilde{N}_t^{-\rho}} \delta_{X_u(t)-m_t^{-\rho}}.
	\end{equation}
	The following theorem gives the convergence of the extremal process.
	We show that the limit of $\e_t^{-\rho}$ is a Poisson random measure with exponential intensity in which each atom is decorated by an independent copy of an auxiliary measure $\D^{\lambda^*}$.
	\begin{theorem}\label{thrm2}
		For any $x>0$ and $\rho < \lambda^*$, we have under $\mathbf{P}_x$
		\begin{equation}
			\lim_{t\rightarrow\infty} \e_t^{-\rho}= \mathrm{DPPP}\left(\lambda^*C_*\Z_{\infty}^{-\rho}e^{-\lambda^*y}\d y, \D^{\lambda^*}\right) \mbox{ in law }
		\end{equation}
		where $C_*$ is as in \eqref{travelling} and $\D^{\lambda^*}$ is defined by \eqref{def_D}.
	\end{theorem}
	
	\begin{remark}
		Theorem \ref{thrm2} is our main result. It states that the logarithmic correction in the median of $\widetilde{M}_t^{-\rho}$ for BBM with absorption is identical to that of classical branching Brownian motion. 
		Additionally, the decorations are also the same. The only difference lies in the intensity of the Poisson point process.
		
		It is worth noting that BBM with absorption has a positive probability of extinction and $\Z_{\infty}^{-\rho}$ has the same probability of being degenerate. Therefore, on the non-extinction event, the limit of the extremal process is non-degenerate. 
	\end{remark}	
	\begin{remark}
		Kesten \cite[Theorem 1]{Kesten78} obtained the asymptotic behaviors of the expectation of the size of $\widetilde{N}_t^{-\rho}$ for different values of $\rho$. 
		Specifically, when $\rho < 0$, $\rho = 0$ or $\rho>0$, the asymptotic behaviors are different. For $\rho>0$, there exists a positive constant $C>0$ such that the expectation of the size of $\widetilde{N}_t^{-\rho}$ is approximately $Ct^{-3/2} e^{\lambda^*t-\frac{\rho^2}{2}t}$ as $t\to\infty$. 
		When $\rho = 0$ (or $\rho < 0$), the expectation of 
		the size of $\widetilde{N}_t^{-\rho}$ is approximately $Ct^{-1/2} e^{\lambda^*t}$ (or $Ce^{\lambda^*t}$) as $t\to\infty$. This suggests that the logarithmic correction of  $m_t^{-\rho}$ might be different depending on $\rho>0$ or not.
		However, Theorem \ref{thrm2} shows that the logarithmic correction of $m_t^{-\rho}$ does not depend on the sign of $\rho$.
	\end{remark}

	As a corollary to the above theorem, we can show that the law of $\widetilde{M}_t^{-\rho}$ converges to the law of a random shifted Gumbel random variable as $t\to\infty$.
	\begin{corollary}\label{corollary3}
		Suppose $x>0$ and $\rho<\lambda^*$. For any $z\in\mathbb{R}$, we have 
		\begin{equation}
			\lim_{t\rightarrow\infty} \mathbf{P}_x(\widetilde{M}_t^{-\rho} - m_t^{-\rho} \leq z) = \mathbf{E}_x(e^{-C_*\Z_{\infty}^{-\rho} e^{-\lambda^*z} }).
		\end{equation}
	\end{corollary}

	Now we briefly describe our strategy for proving the main results. 
	First, we introduce a model equivalent to the BBM with absorption defined above. In this model, the spatial motion is a standard Brownian motion but the absorbing barrier has a drift. There is a close connection between the two models, but the calculations are easier for the latter. 
	Therefore, in this paper, we will prove the equivalent results for the latter. The key to the proof is to guess the precise growth rate of $\widetilde{M}_t^{-\rho}$. 
	\eqref{growth rate} indicates that the asymptotic speed of the right-most particle in the BBM with absorption is $\lambda^*-\rho $ on the survival set. Therefore, as $t\to \infty$, the influence of the absorbing barrier on the right-most particle will continuously decrease. Based on this analysis, we guess that $\widetilde{M}_t^{-\rho}$ still has a logarithmic correction term and the extremal process converges to a decorated Poisson random measure. Only the intensity of Poisson point process is different from that of classical BBM.
	
	In the proof of Theorem \ref{thrm1}, we adapt some idea from \cite{Harris06,Kyprianou04} and prove the convergence of $\Z_t^{-\rho}$ using the method of non-negative supermartingale approximation. Furthermore, based on the ideas in \cite{Harris06}, we show that $\{\Z_{\infty}^{-\rho}>0 \}$ and $\{\zeta^{-\rho} = \infty \}$ agree up to a $\mathbf{P}_x$-null set.
	In the proof of Theorem \ref{thrm2}, we first define the point process $\e^s_t$ using the idea of truncating the absorption barrier at time $s$ and use $\e^s_t$ to approximate the extremal process of the BBM with absorption. Then we can use the results 
	on the extremal process of the classical BBM to obtain the convergence of $\e^s_t$. 
	Next, we use some ideas from \cite{Belloum21,Ma23} to prove that the difference between the Laplace functional of $\e^s_t$ and that of the extremal process tends to $0$ as $t,s\rightarrow\infty$. In this way, we obtain the convergence of the extremal process.
	
	The remainder of the paper is structured as follows. In the next section we introduce a model equivalent to BBM with absorption 
	and state some well-known results on branching Brownian motions. In Section \ref{Sec3}, we prove Theorem \ref{thrm1}. Section \ref{Sec4} is devoted to proving Theorem \ref{thrm2}.
	
	\section{Preliminaries}\label{Sec2}
	\subsection{An equivalent model of BBM with absorption}
	We consider the following BBM with absorption.
	Initially there is a single particle at $x>0$. This particle moves as a standard Brownian motion $B = \{B(t), t\geq 0 \}$ and is killed when it hits the line $\{(y,t):y= \rho t \}$ for some $\rho\in\R$. The particle produces $L$ offspring after an independent exponential time $\eta$ with parameter $\beta$ if 
	it survives up to this moment. 
	We assume that $L$ has distribution  $\{p_k, k\geq 1 \}$ with $\E L\in (1,\infty)$ and $\E L^2<\infty$.
	Starting from their positions of creation, each of these children evolves independently and according to the same law as their parent.
	
	We define a BBM associated to the aforementioned BBM with absorption. When particles hit the line $\{(y,t):y= \rho t\}$, we suppose that they are not killed and evolve as a standard BBM.
	Let $N_t$ be the set of particles of the BBM alive at time $t$.  For any $u\in N_t$ and $s\leq t$, let $X_u(s)$ be the position of the particle $u\in N_t$ or its ancestor at time $s$. Define
	\begin{equation}\label{def_tilde_Nt}
		\widetilde{N}_t := \{u\in N_t: \forall s\leq t, X_u(s) > \rho s \},
	\end{equation}
	then $\widetilde{N}_t$ is the set of particles of this BBM with absorption alive at time $t$. Define 
	\begin{equation}
		Y_t := \sum_{u\in N_t} \delta_{X_u(t)}, \; 
		\widetilde{Y}_t := \sum_{u\in \widetilde N_t} \delta_{X_u(t)} 
	\end{equation} 
	and $\F_t = \sigma(Y_s: s\leq t)$.
	Then $\{Y_t, t\ge 0\}$ and $\{\widetilde{Y}_t, t\ge 0\}$ are point processes describing the number and positions, at time $t$, of individuals of BBM and BBM with absorption respectively.
	We define $\P_x$ as the law of BBM with one initial particle at $x\in\R$, that is $\P_x(Y_0=\delta_x) = 1$.
	We use $\E_x$ to denote the expectation with respect to $\P_x$. For simplicity, $\P_0$ and $\E_0$ will be written as $\P$ and $\E$, respectively. 
	Since $\widetilde{Y}_t$ is a subprocess of $Y_t$, we can still work with the probabilities $\{\P_x:x>0\}$.
	Let $$\zeta:=\inf\{t>0: \widetilde{N}_t = \emptyset \}$$
	be the extinction time of the BBM with absorption and 
	$$\widetilde{M}_t:=\max\{X_u(t):u\in \widetilde{N}_t\}$$
	be the right-most position in the particle system at time $t$.
	Now we compare $\widetilde{Y}_t^{-\rho}$ and $\widetilde{Y}_t$, that is, the BBM with absorption described in Section \ref{Sec1} and \ref{Sec2}. It follows from the definition of $\widetilde{Y}_t$ that the point process $\widetilde{Y}_t$ shifted by $-\rho t$ is given by 
	$$\widetilde{Y}_t-\rho t = \sum_{u\in \widetilde N_t} \delta_{X_u(t)-\rho t}.$$
	It is easy to show that
	\begin{equation}
		\{\widetilde{Y}_t - \rho t, \P_x \} \overset{d}{=} \{\widetilde{Y}_t^{-\rho}, \mathbf{P}_x \}.
	\end{equation}
	Define 
	\begin{equation}\label{def_Z}
		\Z_t := \sum_{u\in \widetilde{N}_t} (\lambda^*t-X_u(t)) e^{\lambda^*(X_u(t)-\lambda^*t) },
	\end{equation}
	then $\{\Z_t, \P_x \} \overset{d}{=} \{\Z_t^{-\rho}, \mathbf{P}_x \}$. Recall that $m_t := \lambda^* t - \frac{3}{2\lambda^*} \log t$ 
	and define the extremal process of $\{\widetilde{Y}_t:t\ge 0\}$ by
	\begin{equation}\label{def_e_t}
		\e_t := \sum_{u\in \widetilde N_t} \delta_{X_u(t)-m_t}.
	\end{equation}
	Based on the above analysis, to prove Theorem \ref{thrm1} and \ref{thrm2}, it is equivalent to show the following theorems.
	\begin{theorem}\label{thrm1'}
		For any $x>0$ and $\rho < \lambda^*$, the limit $\Z_{\infty} := \lim_{t\rightarrow\infty} \Z_t$ exists $\P_x$-almost surely. Moreover, the events $\{\Z_{\infty}>0 \}$ and $\{\zeta = \infty \}$ agree up to a $\P_x$-null set.
	\end{theorem}
	
	\begin{theorem}\label{thrm2'}
		For any $x>0$ and $\rho < \lambda^*$, we have under $\P_x$
		\begin{equation}
			\lim_{t\rightarrow\infty} \e_t= \mathrm{DPPP}\left(\lambda^*C_*\Z_{\infty}e^{-\lambda^*x}\d x, \D^{\lambda^*}\right) \mbox{ in law }
		\end{equation}
		where $C_*$ is as in \eqref{travelling} and $\D^{\lambda^*}$ is defined by \eqref{def_D}.
	\end{theorem}
	In the rest of the paper, we consider the BBM with absorption ${\widetilde{Y}_t}$ and assume $\beta = 1$ and $m=2$ for simplicity. Therefore, $\lambda^* = \sqrt{2}$ and $m_t = \sqrt{2}t - \frac{3}{2\sqrt{2}}\log t$.
	\subsection{Some properties of branching Brownian motion}
	Throughout this paper we use $\{B_t, t\geq 0; \Pi_x \}$ to denote a standard Brownian motion starting from $x$. Expectation with respect to $\Pi_x$ will also be denoted by $\Pi_x$ and $\Pi_0$ will be written as $\Pi$. Let $\{\mathcal{F}_t^B: t\geq 0\}$ be the natural filtration of Brownian motion. For BBMs, the many-to-one lemma (see \cite{Harris17}) is fundamental. Here we state the stopping line version, which can be found in \cite[\S 2.3]{Maillard12}.
	
	For any space-time domain $D$, define $$\tau_D:=\inf\{t\ge 0: (t,B_t)\notin D\}.$$ 
	For any $u\in \cup_{t\geq 0} N_t$, let $\tau_D(u)$ be the stopping time for $\{X_u(t)\}$.
	Define the stopping line 
	\begin{equation}\label{def_stopping_line}
		L_D := \left\{ (u,t) \in (\cup_{t\geq 0} N_t) \times [0,\infty): u\in N_t, \tau_D(u) = t \right\},
	\end{equation}
	and $N_D := \{u\in \cup_{t\geq 0} N_t: (u,t)\in L_D \mbox{ for some } t \}$. We refer to \cite{Chauvin91} for the precise definition.
	\begin{lemma}[Many-to-one Lemma]\label{lemma:many-to-one}
		Let $F:C[0,\infty)\rightarrow\R$ be a bounded measurable function
		such that $F\left(B_s,s\leq \tau_D\right)$ is $\F_{\tau_D}^B$-measurable. 
		Then for any $t>0$
		\begin{equation}\label{many-to-one}
			\E_x \left[\sum_{u\in N_{D}} F\left(X_u(s),s\leq \tau_D(u)\right) \right]  = \Pi_x\left[ e^{\tau_D} F\left(B_s,s\leq \tau_D\right) \right],
		\end{equation}
		where $C[0,\infty)$ denotes the space of continuous functions from $[0,\infty)$ to $\R$.
	\end{lemma}	 
	\begin{remark}
		When $\tau_D = t$, the many-to-one formula above reduces to the following classical many-to-one formula:
		\begin{equation}\label{classical_many-to-one}
			\E_x \left[\sum_{u\in N_t} F\left(X_u(s),s\leq t\right) \right]  = \Pi_x\left[ e^t F\left(B_s,s\leq t\right) \right].
		\end{equation}
	\end{remark}
	
	Now we introduce some results on the additive and derivative martingales of BBM. Define the process
	\begin{equation}
		W_t := \sum_{u\in N_t} e^{\sqrt{2}X_u(t) - 2t},
	\end{equation}
	then $\{W_t, \P_x\}$ is called the additive martingale of BBM starting from $x$. $\{W_t\}$ corresponds to $\{W_t(\lambda)\}$ with $\lambda = \underline{\lambda}$ in \cite{Kyprianou04}. 
	Similarly, the derivative martingale $\{Z_t\}$ given by \eqref{def_derivative} corresponds to $\{\partial W_t(\underline{\lambda})\}$ in \cite{Kyprianou04}.
	By \cite[Theorem 1 and 3]{Kyprianou04}, we have the following lemma.
	\begin{lemma}\label{lemma:add_deri}
		For any $y\in\R$, the limits $W_{\infty} := \lim_{t\rightarrow\infty} W_t$ and $Z_{\infty} := \lim_{t\rightarrow\infty} Z_t$ exist $\P_x$-almost surely. Furthermore, $W_{\infty} = 0$ and $Z_{\infty}\in (0,\infty)$ $\P_x$-almost surely.
	\end{lemma}   
	Note that $Z_t$ can be negative, Kyprianou \cite{Kyprianou04} used a non-negative martingale $\{V_t^z, \P\}$ to approximate $\{Z_t, \P\}$. For any $z>0$, define 
	\begin{equation}\label{def_N_hat}
		\widehat{N}_t^z := \{u\in N_t: \forall s\leq t, X_u(s) < z +\sqrt{2}s \}
	\end{equation}
	and
	\begin{equation}\label{def_mart_V}
		V_t^z := \sum_{u\in\widehat{N}_t^z}  (z+\sqrt{2}t-X_u(t)) e^{\sqrt{2}(X_u(t)-\sqrt{2}t) }.
	\end{equation}
	According to \cite[Theorem 13]{Kyprianou04}, for any $x<z$, the limit $V_{\infty}^z := \lim\limits_{t\rightarrow\infty} V_t^z$ exists $\P_x$-almost surely and is an $L^1(\P_x)$-limit. 
	Define the event
	\begin{equation}\label{def_gamma}
		\gamma^{(z,\sqrt{2})} := \{\forall t\geq 0, \forall u\in N_t, X_u(t) \leq z + \sqrt{2}t \}.
	\end{equation}
	By the proof of \cite[Corollary 10]{Kyprianou04}, we know that on $\gamma^{(z,\sqrt{2})}$,
	\begin{equation}
		V_{\infty}^z = Z_{\infty}
	\end{equation}
	and 
	\begin{equation}\label{eq_gamma_to_1}
		\P(\gamma^{(z,\sqrt{2})}) \uparrow 1 \mbox{ as } z\uparrow\infty.
	\end{equation}
	To prove Theorem \ref{thrm2}, we also need some results about $M_t$ and the extremal process of the classical BBM. 
	The following estimate of the tail probability of $M_t$ can be found in \cite[Corollary 10]{ABK12}.
	\begin{lemma}\label{lemma:M_tail}
		For $y>1$ and $t\geq t_0$, where $t_0$ is a large constant, 
		\begin{equation}
			\P\left( M_t \geq \sqrt{2}t - \frac{3}{2\sqrt{2}}\log t + y \right) \leq b y e^{-\sqrt{2}y - \frac{y^2}{2t} + \frac{3}{2\sqrt{2}} y \frac{\log t}{t}}
		\end{equation}
		for some constant $b>0$.
	\end{lemma}
	
	Let $\mathcal{T}$ be the set of continuous non-negative bounded functions, with support bounded on the left. For any measurable function $f$ and $\sigma$-finite measure $\mu$ on $\R$, we use $\langle f,\mu \rangle$ to denote the integral of $f$ with respect to $\mu$. The following result can be found in \cite[Lemma 4.4]{Berestycki22}. 
	\begin{lemma}\label{lemma:weak_convergence}
		Let $(\mathcal{P}_t,\mathcal{P}_{\infty})$ be point processes on $\R$ with $\mathcal{P}_{\infty}((0,\infty)) < \infty$ a.s. Let $\max \mathcal{P}_t$ ($t\in [0,\infty])$ be the position of the rightmost atom in the point measure $\mathcal{P}_t$. Then the following statements are equivalent: 
		
		$\mathrm{(i)}$ $\lim_{t\rightarrow\infty} \mathcal{P}_t = \mathcal{P}_{\infty}$ and $\lim_{t\rightarrow\infty} \max \mathcal{P}_t = \max \mathcal{P}_{\infty}$ in law.
		
		$\mathrm{(ii)}$ $\lim_{t\rightarrow\infty} (\mathcal{P}_t, \max \mathcal{P}_t) = (\mathcal{P}_{\infty},\max \mathcal{P}_{\infty})$ in law.
		
		$\mathrm{(iii)}$ For all $\varphi\in\mathcal{T}$, $\lim_{t\rightarrow\infty} \E[e^{-\langle \mathcal{P}_t, \varphi\rangle}] = \E[e^{-\langle \mathcal{P}_{\infty}, \varphi\rangle}]$.
	\end{lemma}
	
	The following lemma gives the convergence of Laplace functionals of the extremal process for BBM, see, for example, \cite[Lemma 3.4]{Belloum21}.
	\begin{lemma}\label{lemma:e_t_Laplace}
		For all function $\varphi\in\mathcal{T}$, it holds that
		\begin{equation}
			\lim_{t\rightarrow\infty} \E\left[e^{-\sum_{u\in N_t} \varphi(X_u(t)-m_t) }\right] = \E\left[ \exp\left\{ -C_* Z_{\infty} \int \left(1-\E (e^{-\langle \D^{\sqrt{2}},\varphi(\cdot+z) \rangle })\right) \sqrt{2}e^{-\sqrt{2}z} \d z \right\} \right].
		\end{equation}
	\end{lemma}
	\begin{remark}
		For simplicity, we put
		\begin{equation}\label{def_Cvarphi}
			C(\varphi) = C_* \int \left(1-\E (e^{-\langle \D^{\sqrt{2}},\varphi(\cdot+z) \rangle })\right) \sqrt{2}e^{-\sqrt{2}z} \d z.
		\end{equation}
		By Lemma \ref{lemma:e_t_Laplace}, a simple calculation using a change of variables yields that
		\begin{equation}\label{e_t+y_Laplace}
			\lim_{t\rightarrow\infty} \E[e^{-\sum_{u\in N_t} \varphi(y+X_u(t)-m_t)}] = \E\left[ \exp\left\{ -C(\varphi) Z_{\infty} e^{\sqrt{2}y} \right\}  \right].
		\end{equation}
	\end{remark}

	\section{Proof of Theorem \ref{thrm1}}\label{Sec3}
	In this section, we fix $x>0$ and $\rho<\sqrt{2}$. 
	Recall that $\N_t$ is the set of particles 
	that are alive at time $t$ and has not been absorbed by the line $\{(y,s): y = \rho s\}$ up to time $t$. To prove Theorem \ref{thrm1}, we first prove the following two lemmas.

	\begin{lemma}\label{lemma:mart_W}
		Define
		\begin{equation}\label{def_W_tilde}		
			\widetilde{W}_t := \sum_{u\in \widetilde{N}_t}  e^{\sqrt{2} (X_u(t)-\sqrt{2}t)},
		\end{equation}
		then $\{\widetilde{W}_t, \P_x \}$ is a non-negative supermartingale. Moreover, the limit $\widetilde{W}_\infty := \lim_{t\rightarrow\infty} \widetilde{W}_t$ exists and equals to zero $\P_x$-almost surely.
	\end{lemma}
	\begin{proof}
		For any $s<t$, define
		\begin{equation}\label{def_W^s_tilde}
			\widetilde{W}_t^s := \sum_{v\in \widetilde{N}_s} \sum_{u>v, u\in N_t}  e^{\sqrt{2}(X_u(t)-\sqrt{2}t)},
		\end{equation}
		where the notation $u>v$ means that $u$ is a descendant of $v$. 
		Notice that the set $\N^s_t := \{u\in N_t: \exists v\in \N_s \mbox{ s.t. } u>v \}$ contains all the particles alive at time $t$, which do not hit the line segment $\{(y,r): y = \rho r, 0\leq r\leq s\}$. Hence, $\N_t \subset \N_t^s \subset N_t$ and  $\widetilde{W}_t \leq \widetilde{W}_t^s \leq W_t$. Since $\{W_t, \P_y\}$ is a martingale, we have $\E_y W_t = e^{\sqrt{2} y}$ for any $y\in R$. By the branching property, 
		\begin{align}
			\E_x\left[ \widetilde{W}_t^s \big{|} \F_s \right] 
			&=  \sum_{v\in \widetilde{N}_s} e^{-2s} \E_x\left[ \sum_{u>v, u\in N_t}  e^{\sqrt{2}(X_u(t)-\sqrt{2}(t-s))} \Big{|} \F_s \right]\\
			&= \sum_{v\in \widetilde{N}_s} e^{-2s} \E_{X_v(s)} W_{t-s}(v)\\
			&= \sum_{v\in \widetilde{N}_s} e^{-2s} e^{\sqrt{2} X_v(s)} = \widetilde{W}_s,
		\end{align}
		where given $\F_s$, $\{W_{t-s}(v), v\in\N_s \}$ are independent copies of $W_{t-s}$. Therefore,
		\begin{equation}
			\E_x\left[ \widetilde{W}_t \big{|} \F_s \right] \leq \E_x\left[ \widetilde{W}_t^s \big{|} \F_s \right] = \widetilde{W}_s,
		\end{equation}
		which implies that $\{\widetilde{W}_t, \P_x \}$ is a non-negative supermartingale. Hence the limit $\widetilde{W}_\infty = \lim_{t\rightarrow\infty} \widetilde{W}_t$ exists $\P_x$-almost surely. Since $\widetilde{W}_t\leq W_t$, it follows from Lemma  \ref{lemma:add_deri} that $\widetilde{W}_\infty = 0$ $\P_x$-almost surely. This gives the desired result.
	\end{proof}
	
	Using a similar proof method to Lemma \ref{lemma:mart_W}, we will now prove the convergence of $\Z_t$. Since $\Z_t$ can be negative, the proof of Lemma \ref{lemma:mart_W} is not applicable to $\Z_t$. Therefore, for any $z > x$, we define the following non-negative process:
	\begin{equation}\label{def_V_tilde}
		\V_t^z := \sum_{u\in \N_t \cap \widehat{N}_t^z} (z+\sqrt{2}t-X_u(t)) e^{\sqrt{2}(X_u(t) -\sqrt{2}t)}.
	\end{equation}
	Then we will prove the convergence of $\V_t^z$ using the proof method of Lemma 3.1 and further use $\V_t^z$ to approximate $\Z_t$. Now we have the following lemma.
	\begin{lemma}\label{lemma:mart_V}
		For any $z>x$, $\{\V_t^z, \P_x \}$ is a non-negative supermartingale. The limit $\V_{\infty}^z := \lim_{t\rightarrow\infty} \V_t^z$ exists $\P_x$-almost surely. Moreover, $\V_{\infty}^z$ is non-degenerate.
	\end{lemma}
	\begin{proof}
		First, we show that $\{\V_t^z, \P_x \}$ is a non-negative supermartingale. Recall that $\widehat{N}_t^z$ is defined by \eqref{def_N_hat}. For $s<t$, define
		\begin{equation}\label{def_V_tilde^s}
			\V_t^{z,s} := \sum_{v\in \N_s \cap \widehat{N}_s^z} \sum_{u>v, u\in \widehat{N}_t^z} (z+\sqrt{2}t-X_u(t)) e^{\sqrt{2}(X_u(t) -\sqrt{2}t)}.
		\end{equation}
		By the branching property, we have
		\begin{align}
			\E_x\left[ \V_t^{z,s} \big{|} \F_s \right] &=  \sum_{v\in \N_s \cap \widehat{N}_s^z} \E_x\left[ \sum_{u>v, u\in \widehat{N}_t^z} (z+\sqrt{2}t-X_u(t)) e^{\sqrt{2}(X_u(t) -\sqrt{2}t)} \Big{|} \F_s \right]\\
			&=  \sum_{v\in \N_s \cap \widehat{N}_s^z} \E_x\left[ \sum_{u>v, u\in \widehat{N}_t^z} (z+\sqrt{2}s+\sqrt{2}(t-s)-X_u(t)) e^{\sqrt{2}(X_u(t) -\sqrt{2}(t-s))-2s} \Big{|} \F_s \right]\\
			&= \sum_{v\in \N_s \cap \widehat{N}_s^z} e^{-2s} \E_{X_v(s)} V_{t-s}^{z+\sqrt{2}s}(v),
		\end{align}
		where given $\F_s$, $V_{t-s}^{z+\sqrt{2}s}(v)$ is the counterpart of $V_{t-s}^{z+\sqrt{2}s}$ for the BBM starting from $X_v(s)$. By \cite[Theorem 9]{Kyprianou04}, $\{V_t^z, \P_y\}$ is a martingale and $\E_y V_t^z= (z-y) e^{\sqrt{2}y}$ for any $y<z$. So we have
		\begin{equation}
			\E_x\left[ \V_t^{z,s} \big{|} \F_s \right] = \sum_{v\in \N_s \cap \widehat{N}_s^z} e^{-2s} (z+\sqrt{2}s-X_v(s)) e^{\sqrt{2}X_v(s)} = \V_s^z
		\end{equation}
		and 
		\begin{equation}
			\E_x\left[ \V_t^z \big{|} \F_s \right] \leq \E_x\left[ \V_t^{z,s} \big{|} \F_s \right] = \V_s^z.
		\end{equation}
		Thus, $\{\V_t^z, \P_x \}$ is a non-negative supermartingale and must converge $\P_x$-almost surely.
		
		Next we will prove that the limit $\V_{\infty}^z$ is non-degenerate. Note that $\V_t^z \leq V_t^z$, $\{\V_t^z, \P_x \}$ is a supermartingale and $\{V_t^z, \P_x \}$ is a martingale. Define
		\begin{equation}\label{def_U}
			U_t^z := V_t^z - \V_t^z.
		\end{equation}
		Therefore,
		\begin{equation}
			\E_x[U_t^z | \F_s] = \E_x[V_t^z - \V_t^z | \F_s] \geq V_s^z - \V_s^z = U_s^z.
		\end{equation}
		This implies that $\{U_t^z, \P_x\}$ is a non-negative submartingale. Moreover, the limit
		\begin{equation}\label{eq_U_V_relation}
			\lim_{t\rightarrow\infty} U_t^z = \lim_{t\rightarrow\infty} (V_t^z - \V_t^z) = V_{\infty}^z - \V_{\infty}^z
		\end{equation}
		exists. Let $U_{\infty}^z := \lim_{t\rightarrow\infty} U_t^z$. By \cite[Theorem 13]{Kyprianou04}, $V_{\infty}^z$ is an $L^1(\P_x)$-limit and hence $\E_x V_{\infty}^z = \E_x V_0^z = (z-x)e^{\sqrt{2}x}$. Therefore, to show that $\V_{\infty}^z$ is non-degenerate, it is sufficient to prove that $\E_x U_{\infty}^z < (z-x)e^{\sqrt{2}x}$.
		By the definition \eqref{def_mart_V} and \eqref{def_V_tilde}, we obtain that
		\begin{align}
			\E_x U_t^z = \E_x \left[\sum_{u\in \widehat{N}_t^z,\, u\notin \N_t} (z+\sqrt{2}t-X_u(t)) e^{\sqrt{2}(X_u(t)-\sqrt{2}t)}\right].
		\end{align}
		For any $a,b\in \R$, we define the following two stopping times with respect to Brownian motion:
		\begin{align}
			&\overline{\tau}_{a}^{b} := \inf\{s\geq 0: B_s \geq a+bs \},\\
			&\underline{\tau}_{a}^{b} := \inf\{s\geq 0: B_s \leq a+b s \}. \label{def_tau_rho}
		\end{align}
		Then by the many-to-one formula \eqref{classical_many-to-one}, we have 
		\begin{align}
			\E_x U_t^z &= e^t \Pi_x \left[(z+\sqrt{2}t-B_t) e^{\sqrt{2}(B_t-\sqrt{2}t)} \1_{\{\overline{\tau}_z^{\sqrt{2}}>t, \, \underline{\tau}_{0}^{\rho}\leq t\}} \right]\\
			&= \Pi_x \left[e^{\sqrt{2}(B_t-x)-t} (z+\sqrt{2}t-B_t) e^{\sqrt{2}x}  \1_{\{\overline{\tau}_z^{\sqrt{2}}>t, \, \underline{\tau}_{0}^{\rho}\leq t\}} \right]\\
			&= \Pi_x^{\sqrt{2}} \left[ (z+\sqrt{2}t-B_t) e^{\sqrt{2}x} \1_{\{\overline{\tau}_z^{\sqrt{2}}>t, \, \underline{\tau}_{0}^{\rho}\leq t\}} \right],
		\end{align}
		where the last equality follows from Girsanov's theorem and $\{B_t, \Pi_x^{\sqrt{2}} \}$ is a Brownian motion with drift $\sqrt{2}$ starting from $x$. Let $\widehat{B}_t = \sqrt{2}t - (B_t-x)$, then $\{\widehat{B}_t, \Pi_x^{\sqrt{2}} \}$ is a standard Brownian motion. 
		For any $a,b\in \R$, we define the following two stopping times with respect to $\widehat{B}$:
		\begin{align}
			&\overline{\sigma}_{a}^{b} := \inf\{s\geq 0: \widehat{B}_s \geq a+bs \},\\
			&\underline{\sigma}_{a}^{b} := \inf\{s\geq 0: \widehat{B}_s \leq a+b s \}. 
		\end{align}		
		Therefore,
		\begin{equation}\label{eq_EU_t}
			\E_x U_t^z = e^{\sqrt{2}x} \Pi_x^{\sqrt{2}} \left[ (z-x+\widehat{B}_t) \1_{\{ \underline{\sigma}_{-(z-x)}^{0} > t, \, \overline{\sigma}_x^{\sqrt{2}-\rho}\leq t \}} \right].
		\end{equation}
		Define 
		\begin{equation}\label{eq_meas_change_Bessel}
			\frac{\d \Pi_x^{(\sqrt{2},z)}}{\d \Pi_x^{\sqrt{2}}}\bigg{|}_{\F_t^{\widehat{B}}} = \frac{z-x+\widehat{B}_t}{z-x} \1_{\{ \underline{\sigma}_{-(z-x)}^{0} > t\}},
		\end{equation}
		where $\F_t^{\widehat{B}}$ is the natural filtration of Brownian motion $\{\widehat{B}_t, \Pi_x^{\sqrt{2}} \}$. According to \cite{Imhof84}, $\{z-x+\widehat{B}_t, \Pi_x^{(\sqrt{2},z)}\}$ is a standard Bessel-3 process starting from $z-x$. 
		Combing \eqref{eq_EU_t} and \eqref{eq_meas_change_Bessel}, we obtain that
		\begin{align}
			(z-x)e^{\sqrt{2}x}-\E_x U_t^z &= (z-x)e^{\sqrt{2}x} \left(1- \Pi_x^{(\sqrt{2},z)} \1_{\{\overline{\sigma}_x^{\sqrt{2}-\rho}\leq t \}} \right)\\
			&= (z-x)e^{\sqrt{2}x} \Pi_x^{(\sqrt{2},z)} (\forall s\leq t, z-x+\widehat{B}_s < z + (\sqrt{2}-\rho)s ).
		\end{align} 
		Letting $t\rightarrow\infty$, we have 
		\begin{equation}\label{z-EU^z}
			(z-x)e^{\sqrt{2}x}-\lim_{t\rightarrow\infty} \E_x U_t^z = (z-x)e^{\sqrt{2}x} \Pi_x^{(\sqrt{2},z)} (\forall s\geq 0, z-x+\widehat{B}_s < z + (\sqrt{2}-\rho)s )
		\end{equation}
		Let $\{(B_t^1, B_t^2, B_t^3), \Pi_x^{(\sqrt{2},z)} \}$ be three independent Brownian motions starting at\\ $(\frac{z-x}{\sqrt{3}},\frac{z-x}{\sqrt{3}},\frac{z-x}{\sqrt{3}})$, then it is easy to verify that
		\begin{equation}
			\left\{(z-x+\widehat{B}_t)^2, \Pi_x^{(\sqrt{2},z)} \right\} \overset{d}{=} \left\{ (B_t^1)^2 + (B_t^2)^2 + (B_t^3)^2, \Pi_x^{(\sqrt{2},z)} \right\}.
		\end{equation}
		Thus,
		\begin{align}
			&\Pi_x^{(\sqrt{2},z)} \left(\forall s\geq 0, z-x+\widehat{B}_s < z + (\sqrt{2}-\rho)s \right)\\ 
			\geq& \, \Pi_x^{(\sqrt{2},z)} \left(\forall s\geq 0, \sqrt{3} |B_s^i| < z + (\sqrt{2}-\rho)s, \mbox{ for } i=1,2,3 \right)\\
			= &\, \Pi_x^{(\sqrt{2},z)} \left(\forall s\geq 0, \sqrt{3} |B_s^1| < z + (\sqrt{2}-\rho)s\right)^3. \label{Bessel_BM_estimate}
		\end{align}
		By \cite[Section 3.5.C]{Karatzas}, we know for $y,\mu>0$,
		\begin{equation}
			\Pi (\forall s\geq 0: B_s < y + \mu s) = 1 - e^{-2y\mu}.
		\end{equation}
		Therefore, for standard Brownian motion $\{B_t, \Pi\}$,
		\begin{equation}
			\Pi (\forall s\geq 0: |B_s| < y + \mu s) \geq 1 - 2 \Pi (\exists s\geq 0: B_s \geq y + \mu s) = 1 - 2e^{-2y\mu}
		\end{equation}
		and hence for fixed $\mu>0$ and $y$ large enough, we have 
		$$\Pi (\forall s\geq 0: |B_s| < y + \mu s)>0.$$
		By the Markov property of Brownian motion,
		\begin{align}
			&\Pi_x^{(\sqrt{2},z)} \left(\forall s\geq 0, \sqrt{3} |B_s^1| < z + (\sqrt{2}-\rho)s\right) \geq \Pi\left(\forall s\geq 0,  |B_s| < \frac{x}{\sqrt{3}} + \frac{\sqrt{2}-\rho}{\sqrt{3}}s\right)\\
			&\geq \Pi\left( \1_{\{\max_{s\leq s_0} |B_s| < \frac{x}{\sqrt{3}}\}} \Pi_{B_{s_0}} \left(\forall s\geq s_0, |B_s| < \frac{x}{\sqrt{3}} + \frac{\sqrt{2}-\rho}{\sqrt{3}}s_0 + \frac{\sqrt{2}-\rho}{\sqrt{3}}(s-s_0) \right)\right)\\
			&\geq \Pi\left( \max_{s\leq s_0} |B_s| < \frac{x}{\sqrt{3}}\right) \Pi\left(\forall s\geq 0, |B_s| < \frac{\sqrt{2}-\rho}{\sqrt{3}}s_0 + \frac{\sqrt{2}-\rho}{\sqrt{3}}s \right).
		\end{align}
		Choose $s_0$ large enough that $1-e^{-2s_0(\frac{\sqrt{2}-\rho}{\sqrt{3}})^2} > \frac{1}{2}$. By \cite[Lemma 3.1]{Stroock72}, there exists $\delta>0$ such that $\Pi\left( \max_{s\leq s_0} |B_s| < \frac{x}{\sqrt{3}}\right) \geq \delta$. Combining this with \eqref{z-EU^z} and \eqref{Bessel_BM_estimate}, we get that
		\begin{equation}
			(z-x)e^{\sqrt{2}x} - \lim_{t\rightarrow\infty} \E_x U_t^z \geq (z-x)e^{\sqrt{2}x} (\frac{\delta}{2})^3 > 0.
		\end{equation}
		By Fatou's lemma, 
		\begin{equation}
			\E_x U^z_{\infty} \leq \liminf_{t\rightarrow\infty} \E_x U_t^z < (z-x)e^{\sqrt{2}x}.
		\end{equation}
		This, combined  with \eqref{eq_U_V_relation}, gives that $\E_x \V_{\infty}^z = \E_x V_{\infty}^z - \E_x U^z_{\infty} >0$. Since $\V_{\infty}^z$ is non-negative, we know that $\V_{\infty}^z$ is non-degenerate. This completes the proof.
	\end{proof}
	
	Recall that $\Z_t$, $\gamma^{(z,\sqrt{2})}$, $\widetilde{W}_t$ and $\V_t^z$ are defined by \eqref{def_Z},  \eqref{def_gamma}, \eqref{def_W_tilde} and \eqref{def_V_tilde}, respectively.
	Next, we will give the proof of Theorem \ref{thrm1}. Based on the argument in Section \ref{Sec2}, it is equivalent to proving Theorem \ref{thrm1'}. 
	
	\begin{proof}[Proof of Theorem 1.1]
		Notice that on $\gamma^{(z,\sqrt{2})}$, $N_t = \widehat{N}_t^z$ for any $t>0$. Hence, we have on $\gamma^{(z,\sqrt{2})}$,
		\begin{equation}
			\V_t^z = z\widetilde{W}_t + \Z_t.
		\end{equation}
		Letting $t\rightarrow\infty$, it yields on $\gamma^{(z,\sqrt{2})}$ the limit $\lim_{t\rightarrow\infty}(z\widetilde{W}_t + \Z_t)$ exists and equals to $\V_{\infty}^z$. It follows from Lemma \ref{lemma:mart_W} that $\widetilde{W}_t = 0$ $\P_x$-almost surely. Therefore, we get that on $\gamma^{(z,\sqrt{2})}$,
		\begin{equation}\label{Z_equal_V}
			\lim_{t\rightarrow\infty} \Z_t = \V_{\infty}^z.
		\end{equation}
		Since $\P(\gamma^{(z,\sqrt{2})}) \uparrow 1$ as $z\uparrow \infty$, we know that $\lim_{t\rightarrow\infty} \Z_t$ exists $\P_x$-almost surely and we use $\Z_{\infty}$ to denote this limit.
		
		Next we will prove that $\Z_{\infty}$ is non-degenerate, that is, $\P_x(\Z_{\infty} = 0) < 1$. Notice that for $z_1\leq z_2$, $\widehat{N}_t^{z_1} \subset \widehat{N}_t^{z_2}$, so we have that
		\begin{align}
			\V_t^{z_1} &= \sum_{u\in \N_t \cap \widehat{N}_t^{z_1}} (z_1+\sqrt{2}t-X_u(t)) e^{\sqrt{2}(X_u(t) -\sqrt{2}t)}\\ &\leq \sum_{u\in \N_t \cap \widehat{N}_t^{z_2}} (z_2+\sqrt{2}t-X_u(t)) e^{\sqrt{2}(X_u(t) -\sqrt{2}t)} = \V_t^{z_2}.
		\end{align}
		Letting $t\rightarrow\infty$, this implies $\V_{\infty}^{z_1} \leq \V_{\infty}^{z_2}$ $\P_x$-almost surely. Combining this with \eqref{Z_equal_V} and $\gamma^{(z_1,\sqrt{2})} \subset \gamma^{(z_2,\sqrt{2})}$, we get that $\V_{\infty}^z \leq \Z_{\infty}$ for any $z>0$. By Lemma \ref{lemma:mart_V}, we have $\P_x(\V_{\infty}^z = 0) < 1$ for $z>x$ and hence  $\Z_{\infty}$ is non-degenerate.
		It remains to prove that $\{\Z_{\infty}>0 \}$ and $\{\zeta = \infty \}$ are equivalent up to a $\P_x$-null set.
		Define the function
		$$g(x):=\P_x(\zeta<\infty).$$
		It follows from \cite{Harris06} that $g$ is the unique travelling wave to
		\begin{equation}\label{ODE}
			\left \{
			\begin{aligned}
				& \frac{1}{2}g''-\rho g'+ f(g)-g=0,\quad x>0, \\
				& g(0+)=1,\quad g(\infty)=0.\\
			\end{aligned}\right.
		\end{equation}
		According to the definition of $\Z_t$, it is easy to verify that
		$\{\zeta<\infty\}\subset\{\Z_{\infty}=0\}$. Thus,
		\begin{equation*}
			\P_x(\Z_{\infty}=0)=\P_x(\zeta<\infty)+\P_x(\Z_{\infty}=0;\zeta=\infty).
		\end{equation*}
		So it suffices to show that $\P_x(\Z_{\infty}=0)=\P_x(\zeta<\infty)$. Define $h(x):=\P_x(\Z_{\infty}=0)$. Hence, $h(x)$ satisfies the boundary condition $\lim_{x\to 0^+}h(x)=1$. Since 
		\begin{equation*}
			\begin{aligned}
				\Z_t
				&=\sum_{v\in \N_s}\sum_{u\in \N_t,u>v}\left(\sqrt{2}t-X_u(t)\right)e^{\sqrt{2}(X_u(t)-\sqrt{2}t)}\\
				&= \sum_{v\in\N_s} e^{\sqrt{2}(X_v(s)-\sqrt{2}s)} \sum_{u>v,u\in \N_t} (\sqrt{2}(t-s)-(X_u(t)-X_v(s)))e^{\sqrt{2}((X_u(t)-X_v(s))-\sqrt{2}(t-s))}\\
				&\quad + \sum_{v\in\N_s} e^{\sqrt{2}(X_v(s)-\sqrt{2}s)} \sum_{u>v,u\in \N_t} (\sqrt{2}s+X_v(s)) e^{\sqrt{2}((X_u(t)-X_v(s))-\sqrt{2}(t-s))}\\
				&\overset{d}{=} \sum_{v\in\N_s} e^{\sqrt{2}(X_v(s)-\sqrt{2}s)} \left(\Z_{t-s}(v,X_v(s)-\rho s) + (\sqrt{2}s+X_v(s)) \widetilde{W}_{t-s}(v,X_v(s)-\rho s) \right), 
			\end{aligned}
		\end{equation*}
		where given $\F_s$, $\{(\Z_{t-s}(v,X_v(s)-\rho s), \widetilde{W}_{t-s}(v,X_v(s)-\rho s)), v\in\N_s\}$ are independent copies of $(\Z_{t-s}, \widetilde{W}_{t-s})$ starting from $X_v(s)-\rho s$. By Lemma \ref{lemma:add_deri}, we know that $\lim_{t\rightarrow\infty} \widetilde{W}_{t-s} = 0$ almost surely.
		Letting $t\rightarrow\infty$, it follows that
		\begin{equation*}
			\Z_{\infty}\overset{d}{=}\sum_{v\in \N_s} e^{\sqrt{2}(X_v(s)-\sqrt{2}s)} \Z_{\infty}(v,X_v(s)-\rho s),
		\end{equation*}
		where given $\F_s$, $\{\Z_{\infty}(v,X_v(s)-\rho s), v\in\N_s\}$ are independent copies of $\{\Z_{\infty}, \P_{X_v(s)-\rho s}\}$. Therefore, 
		\begin{equation}\label{g_equation}
			h(x)=\P_x\left(\sum_{v\in \N_s}\Z_{\infty}(v,X_v(s)-\rho s)=0\right)
			=\E_x\left(\prod_{v\in\N_s} h(X_v(s)-\rho s)\right).
		\end{equation}
		It follows from \cite[Proof of Theorem 4]{Harris06} that
		\begin{equation*}
			\frac{1}{2}h''-\rho h'+ f(h)-h=0.
		\end{equation*}
		By \eqref{def_V_tilde} and \eqref{Z_equal_V}, it is easy to show that $h(x)$ is monotone decreasing in $x$. For fixed time $s>0$, by the definition \eqref{def_tilde_Nt}, we have
		\begin{align}
			h(x) &= \E_x \left(\prod_{v\in N_s} [h(X_v(s)-\rho s)]^{\mathbf{1}_{\{\forall r\leq s: X_v(r) \geq \rho r \}}} \right)\\
			&= \E \left(\prod_{v\in N_s} [h(x+X_v(s)-\rho s)]^{\mathbf{1}_{\{\forall r\leq s: x+X_v(r) \geq \rho r \}}} \right) \label{h_equation}
		\end{align}
		For any $v\in N_s$, $\mathbf{1}_{\{\forall r\leq s: x+X_v(r) \geq \rho r \}}\rightarrow 1$ as $x$ tends to infinity.  Letting $x\rightarrow\infty$ on the both side of \eqref{h_equation} and by the bounded convergence theorem, we have
		\begin{equation}
			h(\infty) = \E_x\left( \prod_{v\in N_s} h(\infty) \right).
		\end{equation}
		Hence $h(\infty) = 0$ or $1$. Since $\Z_{\infty}$ is non-degenerate, we have $h(\infty) = 0$. Therefore, $h(x)$ satisfies the equation \eqref{ODE} and by the uniqueness of the one-sided travelling wave, we have $\P_x(\Z_{\infty}=0)=\P_x(\zeta<\infty)$. This completes the proof.
	\end{proof}

	\section{Proof of Theorem \ref{thrm2}}\label{Sec4}
	For any $0<s<t$, define 
	\begin{equation}\label{def_e^s_t}
		\e^s_t := \sum_{v\in\N_s} \sum_{u>v, u\in N_t} \delta_{X_u(t)-m_t}.
	\end{equation}
	The point process $\e^s_t$ will play an important role in the proof of the convergence of $\e_t$. Define
	\begin{equation}\label{def_Z^s_t}
		\Z^s_t := \sum_{v\in\N_s} \sum_{u>v, u\in N_t} (\sqrt{2}t-X_u(t)) e^{\sqrt{2}(X_u(t)-\sqrt{2}t)}.
	\end{equation}
	Recall that
	\begin{equation}\label{def_N^s_t}
		\N^s_t = \{u\in N_t: \exists v\in \N_s, \mbox{ s.t. } u>v \},
	\end{equation}
	then we also have another representation for $\Z^s_t $:
	$$\Z^s_t = \sum_{u\in \N^s_t} (\sqrt{2}t-X_u(t)) e^{\sqrt{2}(X_u(t)-\sqrt{2}t)}.$$ 
	For $t\leq s$, define $\Z^s_t := \Z_t.$
	In the following lemma, we will show that the limit of $\Z^s_t$ as $t\to\infty$ exists $\P_x$-almost surely and is related to the limit of the Laplace functionals of $\e_t^s$.
	\begin{lemma}\label{lemma:e^s_t_Laplace}
		Suppose $x>0$, $\rho<\sqrt{2}$ and $\N_t$ is given by \eqref{def_tilde_Nt}. For any $s>0$,
		\begin{equation}
			\Z_{\infty}^s := \lim_{t\rightarrow\infty} \Z^s_t
		\end{equation}
		exists $\P_x$-almost surely. 	Furthermore, for any non-negative function $\varphi \in \mathcal{T}$, 
		\begin{equation}\label{eq_e^s_t_Laplace}
			\lim_{t\rightarrow\infty} \E_x e^{-\langle \e^s_t, \varphi \rangle} = \E_x\left[ \exp\left\{ -C_* \Z^s_{\infty} \int \left(1-\E (e^{-\langle \D^{\sqrt{2}},\varphi(\cdot+z) \rangle })\right) \sqrt{2}e^{-\sqrt{2}z} \d z \right\} \right].
		\end{equation} 	
	\end{lemma}
	\begin{proof}
		The proof of the existence of $\Z_{\infty}^s$ is similar to that of $\Z_{\infty}$. Recall the definition of $\V_t^{z,s}$  given in \eqref{def_V_tilde^s}. Using the similar argument in the proof of Theorem \ref{thrm1}, 
		we can see that if the limit of $\V_t^{z,s}$ exists as $t\rightarrow\infty$, then on $\gamma^{(z,\sqrt{2})}$,
		\begin{equation}\label{eq_Z^s_t_V}
			\lim_{t\rightarrow\infty} \Z^s_t = \lim_{t\rightarrow\infty} \V_t^{z,s}.
		\end{equation}
		For $s\leq t_1<t_2$, notice that
		\begin{equation}
			\V_{t_2}^{z,s} = \sum_{v\in\N^s_{t_1}\cap\widehat{N}_{t_1}^z} \sum_{u>v, u\in \widehat{N}_{t_2}^z} (z+\sqrt{2}t_2-X_u(t_2)) e^{\sqrt{2}(X_u(t_2) -\sqrt{2}t_2)}.
		\end{equation}
		By the branching property and $\E_y V_t^z = (z-y)e^{\sqrt{2}y}$, we get that
		\begin{align}
			&\E_x\left[ \V_{t_2}^{z,s} \big{|} \F_{t_1} \right] =  \sum_{v\in\N^s_{t_1}\cap\widehat{N}_{t_1}^z} \E_x\left[ \sum_{u>v, u\in \widehat{N}_{t_2}^z} (z+\sqrt{2}t_2-X_u(t_2)) e^{\sqrt{2}(X_u(t_2) -\sqrt{2}t_2)} \Big{|} \F_{t_1} \right]\\
			=&  \sum_{v\in\N^s_{t_1}\cap\widehat{N}_{t_1}^z} e^{-2t_1} \E_x\left[ \sum_{u>v, u\in \widehat{N}_{t_2}^z} (z+\sqrt{2}t_1+\sqrt{2}(t_2-t_1)-X_u(t_1)) e^{\sqrt{2}(X_u(t_2) -\sqrt{2}(t_2-t_1))} \Big{|} \F_{t_1} \right]\\
			=& \sum_{v\in \N_{t_1}^s \cap \widehat{N}_{t_1}^z} e^{-2t_1} \E_{X_v(t_1)} V_{t_2-t_1}^{z+\sqrt{2}t_1}(v)
			= \sum_{v\in \N_{t_1}^s \cap \widehat{N}_{t_1}^z} e^{-2t_1} (z+\sqrt{2}t_1-X_v(t_1)) e^{\sqrt{2}X_v(t_1)} = \V_{t_1}^{z,s},
		\end{align}
		where given $\F_s$, $V_{t_2-t_1}^{z+\sqrt{2}t_1}(v)$ is the counterpart of $V_{t_2-t_1}^{z+\sqrt{2}t_1}$ for the BBM starting from $X_v(t_1)$. Hence $\{\V_{t}^{z,s}, t\geq s, \P_x \}$ is a non-negative martingale and must converge to some limit, say $\V_{\infty}^{z,s}$. By \eqref{eq_gamma_to_1} and \eqref{eq_Z^s_t_V}, we obtain $\Z^s_{\infty} = \lim_{t\rightarrow\infty} \Z^s_t$ exists $\P_x$-almost surely.
		
		Now we will prove \eqref{eq_e^s_t_Laplace}. 
		By \eqref{def_e^s_t} and the branching property and Markov property,
		\begin{align}
			\E_x\left[ e^{-\langle \e^s_t, \varphi \rangle} |\F_s \right] =  \prod_{v\in\N_s} \E_{X_v(s)} \left[ e^{-\sum_{u\in N_{t-s}} \varphi(X_u(t-s)-m_t) } \right].
		\end{align}
		Notice that
		\begin{equation}
			X_u(t-s) - m_t = X_u(t-s) - m_{t-s} - \sqrt{2}s + \frac{3}{2\sqrt{2}}\log \frac{t}{t-s}
		\end{equation}
		and $\lim_{t\rightarrow\infty} \frac{3}{2\sqrt{2}}\log \frac{t}{t-s} = 0$, so by \eqref{e_t+y_Laplace} we get that given the starting point $X_v(s)$
		\begin{equation}
			\lim_{t\rightarrow\infty} \E_{X_v(s)} \left[ e^{-\sum_{u\in N_{t-s}} \varphi(X_u(t-s)-m_t) } \right] =  \E\left[ \exp\left\{ -C(\varphi) Z_{\infty}(v) e^{\sqrt{2}( X_v(s) - \sqrt{2}s)} \right\} | X_v(s) \right],
		\end{equation}
		where given $\F_s$, $\{Z_{\infty}(v),v\in\N_s\}$ are independent copies of $\{Z_{\infty},\P\}$.
		By the bounded convergence theorem, we have
		\begin{align}
			\lim_{t\rightarrow\infty} \E_x e^{-\langle \e^s_t, \varphi \rangle} &= \E_x \left[ \prod_{v\in\N_s} \lim_{t\rightarrow\infty} \E \left[ e^{-\sum_{u>v,u\in N_t} \varphi(X_u(t)-m_t) } | \F_s \right] \right]\\
			&= \E_x\left[ \prod_{v\in\N_s} \E\left[ \exp\left\{ -C(\varphi) Z_{\infty}(v) e^{\sqrt{2}( X_v(s) - \sqrt{2}s)} \right\} | X_v(s) \right] \right] \\
			&= \E_x\left[ \exp\left\{ -C(\varphi) \sum_{v\in\N_s} Z_{\infty}(v) e^{\sqrt{2}( X_v(s) - \sqrt{2}s)} \right\} \right],
		\end{align}
		Since
		\begin{align}
			\Z_t^s =& \sum_{v\in\N_s} e^{\sqrt{2}(X_v(s)-\sqrt{2}s)} \sum_{u>v,u\in N_t} (\sqrt{2}(t-s)-(X_u(t)-X_v(s)))e^{\sqrt{2}((X_u(t)-X_v(s))-\sqrt{2}(t-s))}\\
			&\,+ \sum_{v\in\N_s} e^{\sqrt{2}(X_v(s)-\sqrt{2}s)} \sum_{u>v,u\in N_t} (\sqrt{2}s+X_v(s)) e^{\sqrt{2}((X_u(t)-X_v(s))-\sqrt{2}(t-s))}\\
			\overset{d}{=}& \sum_{v\in\N_s} e^{\sqrt{2}(X_v(s)-\sqrt{2}s)} \left(Z_{t-s}(v) + (\sqrt{2}s+X_v(s)) W_{t-s}(v) \right),
		\end{align}
		where given $\F_s$, $\{(Z_{t-s}(v), W_{t-s}(v)), v\in\N_s\}$ are independent copies of $(Z_{t-s}, W_{t-s})$ for BBM starting from $0$. Letting $t\rightarrow\infty$, it follows from Lemma \ref{lemma:add_deri} that
		\begin{equation}
			\Z_{\infty}^s \overset{d}{=} \sum_{v\in\N_s} e^{\sqrt{2}(X_v(s)-\sqrt{2}s)} Z_{\infty}(v).
		\end{equation}
		Therefore,
		\begin{equation}
			\lim_{t\rightarrow\infty} \E_x e^{-\langle \e^s_t, \varphi \rangle} = \E_x\left[ e^ {-C(\varphi) \Z_{\infty}^s} \right].
		\end{equation}
		This completes the proof.
	\end{proof}
	
	In the following lemma, we show that the limit of $\Z^s_{\infty}$ as $s\to\infty$ exists and is equal to $\Z_{\infty}$. 
	\begin{lemma}\label{lemma:Z^s_converge}
		Suppose $x>0$ and $\rho<\sqrt{2}$. The limit $\lim_{s\rightarrow\infty} \Z^s_{\infty}$ exists $\P_x$-almost surely and is equal to $\Z_{\infty}$.
	\end{lemma}
	\begin{proof}
		By \eqref{Z_equal_V} and \eqref{eq_Z^s_t_V}, we know that on $\gamma^{(z,\sqrt{2})}$, $\Z_{\infty}^s = \V_{\infty}^{z,s}$ and $\Z_{\infty} = \V_{\infty}^z$ almost surely. Therefore, it suffices to show that $\lim_{s\rightarrow\infty} \V_{\infty}^{z,s} = \V_{\infty}^z$. By Lemma \ref{lemma:mart_V} and the proof of Lemma \ref{lemma:e^s_t_Laplace}, $\{\V_{t}^{z}, t\geq 0, \P_x\}$ is a supermartingale and $\{\V_{t}^{z,s}, t\geq s, \P_x\}$ is a non-negative martingale. Combining this with $\V_{t}^{z,s}\geq \V_t^z$, we obtain that $\{\V_{t}^{z,s} - \V_{t}^{z}, t\geq s, \P_x\}$ is a non-negative submartingale and converges to $\V_{\infty}^{z,s} - \V_{\infty}^z$. For any $u\in N_t$, define 
		\begin{equation}\label{def_tau_u}
			\underline{\tau}(u) := \underline{\tau}_{0}^{\rho}(u) = \inf\{s\geq 0: X_u(s) \leq \rho s \}.
		\end{equation} 
		By Fatou's lemma and \eqref{classical_many-to-one}, we have
		\begin{align}
			\E_x[\V_{\infty}^{z,s} - \V_{\infty}^z] &\leq \liminf_{t\rightarrow\infty} \E_x\left[\V_{t}^{z,s} - \V_{t}^{z}\right]\\ 
			&= \liminf_{t\rightarrow\infty} \E_x\left[\sum_{u\in \widehat{N}_t^z} (z+\sqrt{2}t-X_u(t))e^{\sqrt{2}(X_u(t)-\sqrt{2}t)} \1_{\{\underline{\tau}(u)\in(s,t) \}} \right]\\
			&= \liminf_{t\rightarrow\infty} e^t \Pi_x \left[(z+\sqrt{2}t-B_t) e^{\sqrt{2}(B_t-\sqrt{2}t)} \1_{\{\overline{\tau}_z^{\sqrt{2}}>t, \, \underline{\tau}_{0}^{\rho}\in (s,t)\}} \right].
		\end{align}
		Using the similar argument and notation in the proof of Lemma \ref{lemma:mart_V}, we get that
		\begin{align} 
			\E_x[\V_{\infty}^{z,s} - \V_{\infty}^z] &\leq \liminf_{t\rightarrow\infty} e^{\sqrt{2}x} \Pi_x^{\sqrt{2}} \left[ (z-x+\widehat{B}_t) \1_{\{ \underline{\sigma}_{x-z}^{0} > t, \, \overline{\sigma}_x^{\sqrt{2}-\rho}\in (s,t) \}} \right]\\
			&= \liminf_{t\rightarrow\infty} (z-x)e^{\sqrt{2}x}\Pi_x^{(\sqrt{2},z)}\left[\1_{\{\overline{\sigma}_x^{\sqrt{2}-\rho}\in (s,t) \}}\right]\\
			&= (z-x)e^{\sqrt{2}x}\Pi_x^{(\sqrt{2},z)}\left[\1_{\{\overline{\sigma}_x^{\sqrt{2}-\rho}\in (s,\infty) \}}\right],
		\end{align}
		where $\{z-x+\widehat{B}_t, \Pi_x^{(\sqrt{2},z)} \}$ is a standard Bessel-3 process starting from $z-x$ and
		$\overline{\sigma}_x^{\sqrt{2}-\rho} = \inf \{s\geq 0: z-x+\widehat{B}_s \geq z + (\sqrt{2}-\rho)s \}$.
		Therefore, it follows from the bounded convergence theorem that 
		\begin{equation}
			\lim_{s\rightarrow\infty} \E_x[\V_{\infty}^{z,s} - \V_{\infty}^z] \leq	\lim_{s\rightarrow\infty} (z-x)e^{\sqrt{2}x}  \Pi_x^{(\sqrt{2},z)} \left[\1_{\{\overline{\sigma}_x^{\sqrt{2}-\rho}\in (s,\infty) \}}\right] = 0.
		\end{equation}
		Since for any $s_1\leq s_2\leq t$, $\V_{t}^{z,s_1} \geq \V_{t}^{z,s_2}$ $\P_x$-almost surely, it holds that $\V_{\infty}^{z,s_1} \geq \V_{\infty}^{z,s_2}$ $\P_x$-almost surely. So the limit $\lim_{s\rightarrow\infty}\V_{\infty}^{z,s}$ exists $\P_x$-almost surely. By the monotone convergence theorem, we get that
		\begin{equation}
			\E_x[\lim_{s\rightarrow\infty}\V_{\infty}^{z,s} - \V_{\infty}^z] = \lim_{s\rightarrow\infty} \E_x[\V_{\infty}^{z,s} - \V_{\infty}^z] = 0.
		\end{equation}
		Since $\lim\limits_{s\rightarrow\infty}\V_{\infty}^{z,s} - \V_{\infty}^z\geq 0$, we have $\lim\limits_{s\rightarrow\infty}\V_{\infty}^{z,s} = \V_{\infty}^z$ $\P_x$-almost surely. Hence we have 
		$\lim\limits_{s\rightarrow\infty}\Z_{\infty}^{s} = \Z_{\infty}$ on $\gamma^{(z,\sqrt{2})}$. 
		By \eqref{eq_gamma_to_1}, we get the desired result. 
	\end{proof}
	
	In the following lemma, we prove that if a particle hits the line $\{(y,s): y=\rho s\}$ at a large time, then the probability of it having descendants above level $m_t$ at time $t$ approaches zero.
	\begin{lemma}
		For any $A>0$, it holds that
		\begin{equation}\label{eq_tau_geq_s}
			\lim_{s\rightarrow\infty} \limsup_{t\rightarrow\infty} \P_x(\exists u\in N_t: \underline{\tau}(u)\geq s, X_u(t) \geq m_t -A ) = 0.
		\end{equation}
	\end{lemma}
	\begin{proof}
		Let $p\in (0,1)$ and
		\begin{align} 
			I &:= \P_x(\exists u\in N_t: \underline{\tau}(u)> pt, X_u(t) \geq m_t -A ),\\ 
			II &:= \P_x(\exists u\in N_t: \underline{\tau}(u)\in [s,pt], X_u(t) \geq m_t -A ),
		\end{align}
		then $\P_x(\exists u\in N_t: \underline{\tau}(u)\geq s, X_u(t) \geq m_t -A ) = I+II$. Notice that
		\begin{equation}
			I \leq \E_x \left[\sum_{u\in N_t} \textbf{1}_{\{ \underline{\tau}(u)>pt, X_u(t)\geq m_t-A \}}\right].
		\end{equation}
		According to Lemma \ref{many-to-one} and the strong Markov property, we get that
		\begin{align}
			I &\leq \Pi_x[e^t\1_{\{ \underline{\tau}_{0}^{\rho}>pt, \, B_t\geq m_t-A \}}]\\
			&=\Pi_x\left[e^t 1_{\{\underline{\tau}_{0}^{\rho}>pt\} }\Pi_x[\1_{\{ B_t\geq m_t-A \}}|\mathcal{F}_{\underline{\tau}_{0}^{\rho}}] \right]\\
			&=\Pi_x\left[e^t 1_{\{\underline{\tau}_{0}^{\rho}>pt\} }\Pi_{B_{\underline{\tau}_{0}^{\rho}}}[B_{t-\underline{\tau}_{0}^{\rho}}\geq m_t-A] \right],
		\end{align}
		where $\underline{\tau}_{0}^{\rho}$ is defined by \eqref{def_tau_rho}. 
		Since $B_{\underline{\tau}_{0}^{\rho}} = \rho \underline{\tau}_{0}^{\rho}$, we have the following estimate
		\begin{align}
			I &\leq e^t \int_{pt}^t \Pi_x(\underline{\tau}_{0}^{\rho}\in\d r) \Pi_{\rho r}(B_{t-r}\geq m_t-A)\\
			&= e^t\int_{pt}^t  \Pi_x(\underline{\tau}_{0}^{\rho}\in\d r)  \Pi\Big(B_{t-r}\geq \sqrt{2}(t-r) + (\sqrt{2}-\rho)r -\frac{3}{2\sqrt{2}}\log t - A\Big).
		\end{align}
		For any $y\in\R$, an application of Markov's inequality shows that 
		\begin{equation}
			e^{t-r} \Pi(B_{t-r} \geq \sqrt{2}(t-r) + y) \leq e^{-\sqrt{2}y}.
		\end{equation}
		It follows from \cite[Section 3.5.C]{Karatzas} that
		\begin{equation}
			\Pi_x(\underline{\tau}_{0}^{\rho}\in \d r) = \frac{x}{\sqrt{2\pi r^3}} \exp\left\{ - \frac{(x-\rho r)^2}{2r} \right\}
		\end{equation} 
		Therefore,
		\begin{align}
			I &\leq \int_{pt}^t e^r \frac{x}{\sqrt{2\pi r^3}} e^{ - \frac{(x-\rho r)^2}{2r} } e^{-(2-\sqrt{2}\rho)r + \frac{3}{2} \log t + \sqrt{2}A}\d r\\
			&= \int_{pt}^t \frac{x}{\sqrt{2\pi }} \frac{t^\frac{3}{2}}{r^\frac{3}{2}} e^{-(\frac{\rho}{\sqrt{2}}-1)^2r - \frac{x^2}{2r} + x\rho + \sqrt{2}A}\d r\\
			&\leq C \int_{pt}^{t} \frac{1}{p^{\frac{3}{2}}} e^{-(\frac{\rho}{\sqrt{2}}-1)^2r} \d r\\
			&\leq \frac{C}{p^{\frac{3}{2}}} \int_{pt}^{\infty} e^{-(\frac{\rho}{\sqrt{2}}-1)^2r} \d r \overset{t\rightarrow \infty}{\longrightarrow} 0,
		\end{align}
		where $C$ is a positive constant only depending on $x,\rho,A$ and we used $\rho<\sqrt{2}$ here.

		For any $a,b\in \R$, define the space-time domain
		\begin{equation}
			\underline{D}_a^b:=\{(t,x):t> 0, x > a+bt\}.
		\end{equation}
		Recall that the definition of stopping line \eqref{def_stopping_line}. Now we give an estimate of $II$. 
		\begin{equation}
			II \leq \E_x \bigg{[}\sum_{v\in N_{\underline{D}_0^{\rho}}} \1_{\{\underline{\tau}(v)\in [s,pt],\, \exists u>v, u\in N_t, X_u(t) \geq m_t-A \}}\bigg{]}.
		\end{equation}
		By the strong Markov branching property (see Jagers \cite[Theorem 4.14]{Jagers89}, also see Dynkin \cite[Theorem 1.5]{Dynkin91} for the corresponding property for superprocesses, where this property is called the special Markov property),
		\begin{align}
			II &\leq \E_x \Bigg{[}\E_x\bigg{(} \sum_{v\in N_{\underline{D}_{0}^{\rho}}} \1_{\{\underline{\tau}(v)\in [s,pt],\, \exists u>v, u\in N_t, X_u(t) \geq m_t-A \}} | \F_{\underline{\tau}_{0}^{\rho}}\bigg{)}\Bigg{]}\\
			&= \E_x \Bigg{[}\sum_{v\in N_{\underline{D}_{0}^{\rho}}} \1_{\{\underline{\tau}(v)\in [s,pt]\}} \P_{X_v(\underline{\tau}(v))}(M_{t-\underline{\tau}(v)}(v) \geq m_t-A)\Bigg{]},
		\end{align}
		where $\F_{\underline{\tau}_{0}^{\rho}}$ be the natural filtration
		generated by the spatial paths and the number of offspring of the individuals before hitting the stopping line $L_{\underline{D}_{0}^{\rho}}$ and $M_{t-\underline{\tau}(v)}(v) = \max_{u>v, u\in N_t} X_u(t)$. By Lemma \ref{lemma:many-to-one}, we get that
		\begin{align}
			II &\leq \int_s^{pt} e^r \Pi_x(\underline{\tau}_{0}^{\rho}\in \d r) \P_{\rho r}(M_{t-r}(v)\geq m_t-A)\\
			&= \int_s^{pt} e^r \Pi_x(\underline{\tau}_{0}^{\rho}\in \d r) \P(M_{t-r}(v)\geq m_{t-r} + \sqrt{2}r - \frac{3}{2\sqrt{2}}\log \frac{t}{t-r} - A -\rho r).
		\end{align}
		Let $y = (\sqrt{2}-\rho)r - \frac{3}{2\sqrt{2}}\log \frac{t}{t-r} - A$. Since $\sqrt{2}-\rho>0$ and $t/(t-r) < 1/(1-p)$, for $s$ large enough, it holds that $y>1$ for any $r\in[s,pt]$. By Lemma \ref{lemma:M_tail}, we have
		\begin{align}
			II &\leq \int_s^{pt} e^r \frac{x}{\sqrt{2\pi r^3}} e^{ - \frac{(x-\rho r)^2}{2r} } by e^{-\sqrt{2}y-\frac{y^2}{2(t-r)} + \frac{3}{2\sqrt{2}}y\frac{\log (t-r)}{t-r}} \d r\\
			&\leq \int_s^{pt} e^r \frac{x}{\sqrt{2\pi r^3}} e^{ - \frac{(x-\rho r)^2}{2r} } b y e^{-\sqrt{2}y+1} \d r,
		\end{align}
		where the last inequality follows from that  
		$$-\frac{y^2}{2(t-r)} + \frac{3}{2\sqrt{2}}y\frac{\log (t-r)}{t-r} \leq 1$$
		for all $t-r$ large enough. Therefore,
		\begin{align}
			II &\leq C \int_s^{pt} e^r \frac{x}{\sqrt{2\pi r^3}} e^{ - \frac{(x-\rho r)^2}{2r} } r e^{-(2-\sqrt{2}\rho)r + \frac{3}{2}\log \frac{t}{t-r}} \d r\\
			&\leq C \int_s^{pt}\frac{rx}{\sqrt{2\pi r^3}} e^{ - (1-\sqrt{2}\rho+\frac{\rho^2}{2})r - \frac{x^2}{2r} + \frac{3}{2}\log \frac{t}{t-r}} \d r\\
			&\leq C \int_s^{pt} \frac{rx}{\sqrt{2\pi r^3}} (\frac{t}{(1-p)t})^{\frac{3}{2}} e^{-\frac{(x-(\sqrt{2}-\rho)r)^2}{2r}} \d r\\
			&\leq C \Pi_x[\underline{\tau}_{0}^{\sqrt{2}-\rho} \1_{\{\underline{\tau}_{0}^{\sqrt{2}-\rho} \geq s \}} ],
		\end{align}
		where the positive constant $C$ is changed line by line and only depends on $x,\rho,A$. By \cite[Exercise 3.5.10]{Karatzas}, a direct calculation shows that $\Pi_x[\underline{\tau}_{0}^{\sqrt{2}-\rho}] = \frac{x}{\sqrt{2}-\rho}$. Therefore, by the dominated convergence theorem, 
		\begin{equation}
			\lim_{s\rightarrow\infty} \Pi_x[\underline{\tau}_{0}^{\sqrt{2}-\rho} \1_{\{\underline{\tau}_{0}^{\sqrt{2}-\rho} \geq s \}} ] = 0.
		\end{equation}
		Thus,
		\begin{equation}
			\lim_{s\rightarrow\infty} \limsup_{t\rightarrow\infty} \P_x(\exists u\in N_t: \underline{\tau}(u)\in [s,pt], X_u(t) \geq m_t -A ) = 0.
		\end{equation}
		This completes the proof.
	\end{proof}
	
	Based on the above lemmas, now we  present the proof of Theorem \ref{thrm2}.
	\begin{proof}[Proof of Theorem \ref{thrm2}]
		It is equivalent to prove Theorem \ref{thrm2'}.
		Let $A$ be chosen such that $\mathrm{supp}(\varphi) \subset [-A,\infty)$. Note that only on the event $\{\exists u\in N_t: \underline{\tau}(u)\geq s, X_u(t) \geq m_t -A \}$, the point processes $\e^s_t$ and $\e_t$ are not equal. Therefore,
		\begin{equation}\label{eq_e^s_bound}
			|\E_x (e^{-\langle \e^s_t, \varphi \rangle}) - \E_x (e^{-\langle \e_t, \varphi \rangle})| \leq \P_x(\exists u\in N_t: \underline{\tau}(u)\geq s, X_u(t) \geq m_t -A).
		\end{equation}
		By Lemma \ref{lemma:Z^s_converge} and the bounded convergence theorem, it holds that
		\begin{equation}\label{eq_Z^s_Laplace_converge}
			\lim_{s\rightarrow\infty} \E_x(e^{-C(\varphi) \Z^s_{\infty}}) = \E_x(e^{-C(\varphi) \Z_{\infty}}),
		\end{equation}
		where $C(\varphi)$ is given by \eqref{def_Cvarphi}.
		By \eqref{eq_e^s_t_Laplace}, \eqref{eq_tau_geq_s}, \eqref{eq_e^s_bound} and \eqref{eq_Z^s_Laplace_converge}, we get that
		\begin{align}
			\limsup_{t\rightarrow\infty} &|\E_x (e^{-\langle \e_t, \varphi \rangle}) - \E_x(e^{-C(\varphi) \Z_{\infty}})| 
			\leq \lim_{s\rightarrow\infty}\limsup_{t\rightarrow\infty} |\E_x (e^{-\langle \e^s_t, \varphi \rangle}) -\E_x (e^{-\langle \e_t, \varphi \rangle})|\\ 
			&+ \limsup_{t\rightarrow\infty} |\E_x (e^{-\langle \e^s_t, \varphi \rangle}) - \E_x(e^{-C(\varphi) \Z^s_{\infty}})| + \lim_{s\rightarrow\infty} |\E_x(e^{-C(\varphi) \Z^s_{\infty}}) - \E_x(e^{-C(\varphi) \Z_{\infty}})|\\
			=& \,0.
		\end{align}
		By the definition of $C(\varphi)$ and Campbell formula for Poisson random measure, a direct calculation shows that the Laplace functionals of DPPP$\left(\sqrt{2}C_*\Z_{\infty}e^{-\sqrt{2}x}\d x, \D^{\sqrt{2}}\right)$ is $\E_x(e^{-C(\varphi) \Z_{\infty}})$. This completes the proof.
	\end{proof}
	
	At the end of the paper, we provide a proof of Corollary \ref{corollary3} by combining the convergence of the extremal process and the property of $\D^{\sqrt{2}}$.
	\begin{proof}[Proof of Corollary \ref{corollary3}]
		Define 
		\begin{equation}
			\e^{-\rho}:=\mathrm{DPPP}\left(\sqrt{2}C_*\Z_{\infty}^{-\rho}e^{-\sqrt{2}z}\d z, \D^{\sqrt{2}}\right),
		\end{equation}
		and let $\mathcal{Q}^{-\rho}$ be the Poisson point process with intensity $\sqrt{2}C_*\Z_{\infty}^{-\rho}e^{-\sqrt{2}z}\d z$.
		According to Theorem \ref{thrm2} and Lemma \ref{lemma:weak_convergence}
		\begin{equation}
			\lim_{t\to\infty}\max\e_t^{-\rho}=\max\e^{-\rho}\ \mbox{ in law }.
		\end{equation}
		Since decoration $\D^{\sqrt{2}}$ has no effect on the maximum of the limit of extremal process, it follows that 
		\begin{align}
			\lim_{t\to\infty}\mathbf{P}_x\left(\widetilde{M}_t^{-\rho}-m_t^{-\rho}\le z\right)
			&=\mathbf{P}_x\left(\e^{-\rho}(z,\infty)=0\right)\\
			&=\mathbf{P}_x\left(\mathcal{Q}^{-\rho}(z,\infty)=0\right)\\
			&=\mathbf{E}_x\left[\mathbf{E}_x\left[e^{-\int_{z}^{\infty}\sqrt{2}C_*\Z_{\infty}^{-\rho}e^{-\sqrt{2}y}\d y }|\Z_{\infty}^{-\rho}\right]\right]\\				&=\mathbf{E}_x\left[e^{-C_*e^{-\sqrt{2}z} \Z_{\infty}^{-\rho} }\right].
		\end{align}
		The proof is completed.
	\end{proof}
	
	\vspace{.1in}
	\textbf{Acknowledgment}:
	We would like to express deep gratitude to Professor Yanxia Ren and Professor Renming Song for their constructive suggestions.
	
\end{document}